\documentclass[reqno,11pt]{amsart}

\usepackage[dvips,pdftex]{epsfig}
\usepackage{amssymb}
\usepackage{amsmath}
\usepackage{amscd}
\usepackage{amsthm}
\usepackage[mathscr]{eucal}
\usepackage[all]{xy}
\usepackage{color}

\newtheorem{theorem}{Theorem}
\newtheorem{lemma}{Lemma}
\newtheorem{corollary}{Corollary}

\theoremstyle{remark}
\newtheorem{remark}{Remark}

\theoremstyle{definition}
\newtheorem{definition}{Definition}

\DeclareMathOperator{\elm}{elm}

\DeclareMathOperator{\Jac}{Jac}
\DeclareMathOperator{\id}{id}
\DeclareMathOperator{\Ab}{Ab}
\DeclareMathOperator{\dAb}{dAb}

\DeclareMathOperator{\E}{E}

\DeclareMathOperator{\HH}{H}
\DeclareMathOperator{\K}{K}
\DeclareMathOperator{\TT}{T}

\DeclareMathOperator{\WW}{W}
\DeclareMathOperator{\bl}{Bl}

\DeclareMathOperator{\End}{End}

\DeclareMathOperator{\pic}{Pic}

\begin{document}

\title[Tangency and a ruled surface]{Tangency and a ruled surface associated with a Hitchin system}
\author[Taejung Kim]{Taejung Kim}

\address{Korea Institute for Advanced Study\\
207-43 Cheongyangri-dong\\
Seoul 130-722, Korea}

\thanks{The author would like to express his sincere gratefulness to the Korea Institute for Advanced Study for providing him a hospital environment
while preparing this manuscript and to William Goldman, Dongseon Hwang, Bumsig Kim, Hoil Kim, Junmyeong Jang, Byungheup Jun, Serguei Novikov, and Niranjan Ramachandran for giving him their deep insights, lively discussions, and valuable suggestions about this paper.}

\keywords{Elementary transformation, Elliptic soliton, Hitchin system, Linear system,  Residue section, Ruled surface, Tangential cover.}

\subjclass[2010]{14C20, 14H60, 14H70, 14J26,  32L10}
\date{\today}
\email{tjkim@kias.re.kr}

\begin{abstract}
We will generalize the Treibich-Verdier theory about elliptic solitons to a Hitchin system by constructing a particular ruled surface and we will propose a generalization of a tangency condition associated with elliptic solitons to a Hitchin system. In particular, we will calculate the dimension of the moduli space of Hitchin covers satisfying the tangency condition. 
With this new point of view, we will see a subtle relation between the characterizations of coverings and the singularities of divisors in a particular algebraic surface.
\end{abstract}
\maketitle

\section{Introduction and preliminaries}\label{intropre}

The origin of a tangential cover is rooted in the investigation of the solutions of the Kd-V equation. More specifically, it is originated from the reduction theory, i.e., how to reduce a theta function of a given genus to theta functions of lower genera. The theta function is in some sense a muti-dimensional Fourier transform. It is not easy to handle the expression. So, it was a very active area of nineteen century to invent the method of how to reduce it.

It is well known that the Kd-V equation have an explicit theta solution, so-called ``Mateev-Its formula'' by the work of the Russian school. So, the modern reduction theory around the Kd-V theory, more generally non-linear evolution equations, has dealt with how to express the theta formula in terms of an elliptic function, which is the origin of the terminology, \emph{elliptic soliton}. An elliptic soliton is, simply speaking, a solution of the K-P equation, more generally any nonlinear evolution equation, which can be written as elliptic functions. Krichever gave an explicit formula of an elliptic soliton associated with the K-P equation by generalizing the work of Airault, McKean, and Moser about the Kd-V equation. The main ingredients in Krichever's work \cite{kr80} were a Lax representation and a Calogero-Moser system. When the concept of a tangential cover, which was apparent in the work of Krichever in elliptic soliton with hindsight, was first introduced by Treibich and Verdier, this dynamical system point of view essentially becomes the realm of algebro-geometric problems. Hence, it is very obvious, at least to the author, that translating dynamical behaviors to algebrogeometric tools should be an interesting work to do in answering many problems in non-linear evolution equations as well as the other directions.

Hence, in this paper we will generalize the Treibich-Verdier theory about elliptic solitons to a Hitchin integrable system. The main framework is two-fold: The first is to construct a particular ruled surface associated with a Hitchin system \cite{hit87, don96} which generalizes a ruled surface in the Treibich-Verdier theory. In \cite{tv90}, Treibich and Verdier construct a surface $S$ which is a projectivization of rank $2$ bundle $\WW$ over an elliptic curve. It is a universal embedding bundle of a principal affine-bundle $\Delta$ over $X$:
$$0\to\mathbb{G}_a\to\Delta\to X\to 0.$$
In order to extend the Treibich-Verdier theory about an elliptic curve to a general algebraic curve of an arbitrary genus, we need to construct a particular ruled surface whose role substitutes the role of $S$ in \cite{tv90}. We will deal with this matter in Section~\ref{sec;uni}. The second is to generalize the tangency condition associated with elliptic curves in \cite{tv90} to appropriate one in a Hitchin system. In \cite{tv90}, Treibich and Verdier call a pointed morphism $\pi:(\widehat{\mathfrak{R}},p)\to (\mathfrak{R},q)$ a \emph{tangential cover} if $\pi^\ast \circ \dAb(\TT_{q}\mathfrak{R})$ is tangent to $\Ab(\widehat{\mathfrak{R}})$ at $p$ where $\TT_{q}\mathfrak{R}$ is a tangent space. Here $\Ab$ is the Abel map $\Ab:\widehat{\mathfrak{R}}\to\Jac(\widehat{\mathfrak{R}})$ or $\Ab:\mathfrak{R}\to\Jac(\mathfrak{R})$ and $\dAb$ is the associated tangential map. In \cite{tv90}, the authors confine themselves to the case when $\mathfrak{R}$ is an elliptic curve. Of course, it looks tangible that the concept of tangency makes sense regardless of genus of a compact Riemann surface $\mathfrak{R}$. On the other hands, it seems unclear whether or not this can give interesting results when $\mathfrak{R}$ is replaced to a general Riemann surface of genus $>1$. Suppose that $\pi:\widehat{\mathfrak{R}}\to\mathfrak{R}$ is a Hitchin spectral cover \cite{hit87, don96}. This spectral curve is a natural realization of a Higgs bundle $\phi$: A Higgs field is a section of $\End\E\otimes\K_\mathfrak{R}$ where $\E$ is a holomorphic vector bundle of rank $n$ and $\K_\mathfrak{R}$ is a canonical bundle of a compact Riemann surface $\mathfrak{R}$. Then a Hitchin spectral curve is defined by
$$\widehat{\mathfrak{R}}=\{\det(\chi\cdot \id-\phi)=0 \}$$
where $\chi$ is a tautological section of $\pi^{\ast}_{\K_\mathfrak{R}}\K_\mathfrak{R}$ and $\id$ is the identity map in $\End\E$. Schematically we see
$$\xymatrix{\K_\mathfrak{R}\ar[r]^{\pi^{\ast}_{\K_\mathfrak{R}}}\ar[d]^{\pi_{\K_\mathfrak{R}}} &\pi^{\ast}_{\K_\mathfrak{R}}\K_\mathfrak{R}\\
\mathfrak{R} &     \widehat{\mathfrak{R}}\subset \K_\mathfrak{R}\ar[l]^(0.6)\pi\ar[u]_\chi.
}$$
Even though the tangency condition in \cite{tv90} makes sense for a Riemann surface of arbitrary genus, we claim that a trivial straightforward generalization of this concept would not work, since the Treibich and Verdier theory \cite{tv90,trei97} heavily uses the fact that $\dim\HH^0(X,\mathcal{O}_X)=1$ where $X$ is an elliptic curve. Hence, a suitable modification of this concept is necessary to get more interesting theory. Therefore we will modify the concept of the tangency and show why this modification is indeed a right generalization of the tangency for an elliptic curve, which will justify the investigation of this paper, we hope.

The structure of this paper is as follows: In Section~\ref{sec;uni}, we will construct a new ruled surface where Hitchin spectral curves can be defined as divisors. This new surface $\mathfrak{S}$ will take the place of the projectivized cotangent bundle $\mathbb{P}(\K_\mathfrak{R}\oplus\mathbb{C})$ of a compact Riemann surface $\mathfrak{R}$ in \cite{hit87} as well as generalizing the role of $S$ in the elliptic soliton theory. In Section~\ref{seclin} we will characterize the properties of Hitchin divisors in $\mathfrak{S}'=\mathbb{P}(\K_\mathfrak{R}\oplus\mathbb{C})$ and describe a linear system of them. In Section~\ref{sechit}, we will study Hitchin spectral curves which can be realized as divisors in  $\mathfrak{S}$ and indicate a necessary condition for a Hitchin cover to become a divisor in $\mathfrak{S}$. Moreover, we will also describe the moduli space consisting of such Hitchin covers. In Section~\ref{secpol}, we will generalize the tangency condition in the Treibich-Verdier theory to a Hitchin system and we will show that this condition indeed describes the Hitchin divisors in $\mathfrak{S}$. Once we have established the basic necessary frameworks, we will characterize an implication of the tangency condition to the defining equation of a Hitchin spectral curve. Finally, we will announce other investigations in the future using the theory of this paper.

\section{Ruled surface and associated vector bundle}\label{sec;uni}

In this section we will generalize the construction of the ruled surface in Treibich-Verdier theory about elliptic solitons to a Hitchin system. For the backgrounds of basic facts about this section, we refer to \cite{har77,maru70}: An \emph{elementary transformation} of a sheaf $\mathcal{W}$ of vector bundle $\WW$ of rank $2$ over a Riemann surface $\mathfrak{R}$ associated with a surjective morphism $u_q$ for $q\in\mathfrak{R}$ is
$$\xymatrix{
0\ar[r]&\elm_{u_q}(\mathcal{W} )\ar[r]&\mathcal{W}\ar[r]^{u_q}&\mathbb{C}_q\ar[r]&0.}$$
For example, letting $u_q$ be a projection to the first factor of $\mathcal{O}_\mathfrak{R}\oplus\mathcal{O}_\mathfrak{R}$, we have
$$\elm_{u_q}(\mathcal{O}_\mathfrak{R}\oplus\mathcal{O}_\mathfrak{R} )=\mathcal{O}_\mathfrak{R}(-q)\oplus\mathcal{O}_\mathfrak{R}.$$
By projectivization, we see
\begin{equation}\label{eqpro1}
\mathbb{P}\Big(\elm_{u_q}(\mathcal{O}_\mathfrak{R}\oplus\mathcal{O}_\mathfrak{R})\Big)=\mathbb{P}(\mathcal{O}_\mathfrak{R}(-q)\oplus\mathcal{O}_\mathfrak{R})=\mathbb{P}(\mathcal{O}_\mathfrak{R}\oplus\mathcal{O}_\mathfrak{R}(q)).
\end{equation}
In terms of a projectivized bundle, the corresponding process of the above is given by $\elm_p\Big(\mathbb{P}^1\times\mathfrak{R}\Big)$ where $p=([1,0],q)$. Geometrically, this process is the blowing up at $p$ followed by the contraction of the original fiber $\pi^{-1}(q)$ where $\pi:\mathbb{P}^1\times\mathfrak{R}\to\mathfrak{R}$.

M. Atiyah showed that a $\mathbb{P}^1$-bundle over a complete non-singular curve is represented by a vector bundle of rank $2$ with the set of transition functions  $\{G_{ij}(q)=\begin{pmatrix}
a_{ij}(q)&b_{ij}(q)\\
0&c_{ij}(q)\end{pmatrix}\}$ (see \cite{maru70} for details). In this notation, 
we calculate explicitly the set of transition functions of a vector bundle of rank $2$ after certain elementary transformations.

\begin{theorem}\label{ele1}
Let $\{(U_i,z_i)\}$ be an open cover with local coordinates $z_i$ of $\mathfrak{R}$, $p'=(\infty,q_0)=([1,0],q_0)$, and $p=([-1:1],q_0)$.
The transition function $G_{ij}\in\mathbf{PGL}(2,\mathbb{C})$ on $U_i\cap U_j$ of a projective bundle $\elm_{p}\circ\elm_{p'}\Big(\mathbb{P}^1\times\mathfrak{R}\Big)$ is given by
$$G_{ij}(q)=\Big[\begin{pmatrix}
1&(g_{ij}(q)-1)z_i(q)\\
0&1\end{pmatrix}\Big]\in\mathbf{PGL}(2,\mathbb{C})$$
Here $\{g_{ij}(q)=\frac{z_j(q)}{z_i(q)}\}$ is the set of the transition functions of a line bundle $\mathcal{O}_\mathfrak{R}(q_0)$ and
$g_{ij}(q)=1$ for $q\in U_i\cap U_j$ such that $q_0\not\in U_i\cap U_j$.
\end{theorem}

\begin{proof}
  Consider a trivial bundle $\mathbb{C}^2\times\mathfrak{R}$ of rank $2$ and its projectivization $\mathbb{P}^1\times\mathfrak{R}$. By an elementary transformation (see \cite{maru70}) at $p'=(\infty,q_0)=([1,0],q_0)$, a basis $\{e_1,e_2\}$ of global sections in $\HH^0(\mathfrak{R},\mathcal{O}_\mathfrak{R}\oplus\mathcal{O}_\mathfrak{R})$ is transformed to $\{ze_1,e_2\}=\{e'_1,e'_2\}$ around $q_0$ with a local coordinate $z$, i.e., $ae_1+be_2\mapsto\frac{a}{z}e'_1+be'_2$.
An equivalent procedure up to projectivization in the spirit of Equation~\eqref{eqpro1} is to add zeros at $q_0$, i.e.,
$$ae_1+be_2\mapsto ae'_1+zbe'_2.$$
Consequently, the defining transformation at $q_0$ of $\elm_{p'}\Big(\mathbb{P}^1\times\mathfrak{R}\Big)$ is given by $\begin{pmatrix}\frac{1}{z}&0\\0&1\end{pmatrix}$. On the other hands, by an elementary transformation at $p=([-1,1],q_0)$, we have
$$ae_1+be_2\mapsto (a+b)e'_1+\frac{b}{z}e'_2.$$
That is, a basis $\{e_1,e_2\}$ of global sections in $\HH^0(\mathfrak{R},\mathcal{O}_\mathfrak{R}\oplus\mathcal{O}_\mathfrak{R})$ is transformed to $\{e_1,z(e_2-e_1)\}=\{e'_1,e'_2\}$ around $q_0$ with a local coordinate $z$. Consequently, the transition function at $q_0$ of $\elm_{p}\Big(\mathbb{P}^1\times\mathfrak{R}\Big)$ is given by $\begin{pmatrix}1&1\\0&\frac{1}{z}\end{pmatrix}$. Combining those transformations,  the defining transformation at $q_0$ of $\elm_p\circ\elm_{p'}\Big(\mathbb{P}^1\times\mathfrak{R}\Big)$ is given by $\begin{pmatrix}\frac{1}{z}&1\\0&\frac{1}{z}\end{pmatrix}$. Then the transition function $G_{ij}$ on $U_i\cap U_j$ is given by
$$\begin{pmatrix}\frac{1}{z_j}&1\\0&\frac{1}{z_j}\end{pmatrix}\begin{pmatrix}z_i&-z_{i}^{2}\\0&z_i\end{pmatrix}=\begin{pmatrix}\frac{z_i}{z_j}&\frac{z_i}{z_j}(z_j-z_i)\\0&\frac{z_i}{z_j}\end{pmatrix}.$$
Since $g_{ij}=\frac{z_j}{z_i}$ is a transition function of $\mathcal{O}_\mathfrak{R}(q_0)$, as a projective transformation it is the same as
$$G_{ij}(q)=\begin{pmatrix}1&(g_{ij}(q)-1)z_i(q)\\0&1\end{pmatrix}.$$
\end{proof}

Let $K_\mathfrak{R}=\sum_{i=1}^{2g-2}q_i$ be a canonical divisor of $\mathfrak{R}$. In particular, let us assume all the $q_i$ are distinct throughout the paper unless otherwise specified. Let $p'_i=(\infty,q_i)$ for $i=1,\dots,2g-2$ and $p_i=([-1,1],q_i)$ for $i=1,\dots,2g-2$. Note that we denote $\infty=[1,0]$. By Theorem~\ref{ele1}, we have the following.

\begin{corollary}\label{elco1}
Let $\{(U_i,z_i)\}$ be an open cover with local coordinates $z_i$ of $\mathfrak{R}$. The transition function $G_{ij}\in\mathbf{PGL}(2,\mathbb{C})$ on $U_i\cap U_j$ of a projective bundle

$$\elm_{p_1}\circ\dots\circ \elm_{p_{2g-2}}\circ\elm_{p'_1}\circ\dots\circ \elm_{p'_{2g-2}} \Big(\mathbb{P}^1\times\mathfrak{R}\Big)$$
is given by
$$G_{ij}(q)=\Big[\begin{pmatrix}
1&(g_{ij}(q)-1)z_i(q)\\
0&1\end{pmatrix}\Big]\in\mathbf{PGL}(2,\mathbb{C})$$
Here $\{g_{ij}(q)=\frac{z_j(q)}{z_i(q)}\}$ is the set of the transition functions of a line bundle $\mathcal{O}_\mathfrak{R}(K_\mathfrak{R})$.
\end{corollary}

We will denote
$$\mathfrak{S}:=\elm_{p_1}\circ\dots\circ \elm_{p_{2g-2}}\circ\elm_{p'_1}\circ\dots\circ \elm_{p'_{2g-2}} \Big(\mathbb{P}^1\times\mathfrak{R}\Big).$$
From Corollary~\ref{elco1}, we deduce that $\mathfrak{S}$ is a projectivization of a vector bundle $\WW$ of rank $2$ with a set $\Big\{\begin{pmatrix}
1&(g_{ij}(q)-1)z_i(q)\\
0&1\end{pmatrix}\Big\}$ of transition functions. In particular, from the explicit expression of the transition functions we can see that $\bigwedge^2\WW$ is trivial and has a trivial sub-bundle. Hence, it defines an extension class $[\alpha]\in\HH^1(\mathfrak{R},\mathcal{O}_\mathfrak{R})$:
$$[\alpha]:0\to\mathcal{O}_\mathfrak{R}\to\WW\to\mathcal{O}_\mathfrak{R}\to0.$$

\begin{remark}
This construction is the generalization of the construction of a surface $S$ in \cite{tv90} which is a projectivization of rank $2$ bundle $\WW$ over an elliptic curve.
The bundle $\WW$ is an universal embedding bundle of a principal affine-bundle $\Delta$ over $X$:
$$0\to\mathbb{G}_a\to\Delta\to X\to 0.$$
\end{remark}

Let $C_0$ be a section of $\pi_S:\mathfrak{S}\to \mathfrak{R}$ corresponding to a trivial sub-bundle. It is obvious that $\HH^0(\mathfrak{R},\mathcal{W})\ne0$ where $\mathcal{W}$ is the sheaf of $\WW$. Moreover, $\HH^0(\mathfrak{R},\mathcal{W}\otimes\mathcal{L})=0$ for any line bundle $\mathcal{L}$ with negative degree from the induced long exact sequence of the following short exact sequence and the fact $\HH^0(\mathfrak{R},\mathcal{L})=0$:
$$0\to\mathcal{L}\to\mathcal{W}\otimes\mathcal{L}\to\mathcal{L}\to0.$$
That is, the sheaf $\mathcal{W}$ is \emph{normalized} (see \emph{p.373} in \cite{har77}). Consequently,
$$C_0.C_0=qf.qf=0\text{ and }C_0.qf=1\text{ where }$$
$qf$ is a divisor which is the fiber $\pi^{-1}_{\mathfrak{S}}(q)$ of $\pi_\mathfrak{S}:\mathfrak{S}\to\mathfrak{R}$. Moreover, from \emph{p.373} in \cite{har77} we see that the canonical divisor is given by
$$K_\mathfrak{S}\sim -2C_0+K_\mathfrak{R}f$$
where $K_\mathfrak{R}$ is a canonical divisor of $\mathfrak{R}$.

\section{Hitchin covers in a linear system in $\mathfrak{S}'=\mathbb{P}(\K_\mathfrak{R}\oplus\mathbb{C})$.}\label{seclin}
A Hitchin spectral curve $\pi:\widehat{\mathfrak{R}}\to \mathfrak{R}$ of degree $n$ over a compact Riemann surface $\mathfrak{R}$ of genus $g$ is defined by a zero divisor of a section $s$ of a line bundle $\pi^{\ast}_{\K_\mathfrak{R}}\K_{\mathfrak{R}}^{n}$ over a non-compact complex surface $\K_\mathfrak{R}$ where $\pi_{\K_\mathfrak{R}}:\K_\mathfrak{R}\to\mathfrak{R}$. Since the line bundle $\pi^{\ast}_{\K_\mathfrak{R}}\K_{\mathfrak{R}}^{n}$ has a section $\chi^n$ where $\chi$ is a tautological section of  $\pi^{\ast}_{\K_\mathfrak{R}}\K_{\mathfrak{R}}$ over $\K_\mathfrak{R}$, we deduce that a Hitchin spectral curve $\widehat{\mathfrak{R}}$ is linearly equivalent to the zero divisor $n\mathfrak{R}$ of $\chi^n$. On the other hands, since $s$ is a section of $\pi':\pi^{\ast}_{\K_\mathfrak{R}}\K_{\mathfrak{R}}^{n}\to\K_\mathfrak{R}$, we also infer that a Hitchin spectral curve $\widehat{\mathfrak{R}}$ is linearly equivalent to a divisor $\pi'^{-1}(nK_\mathfrak{R})$ where $K_\mathfrak{R}$ is a canonical divisor of $\mathfrak{R}$. In particular, from the adjunction formula and the triviality of canonical bundle $\K_{\K_\mathfrak{R}}$ of the non-compact space $\K_\mathfrak{R}$, we have
$$g(\widehat{\mathfrak{R}})=n^2(g-1)+1\text{ and }\dim|n\mathfrak{R}|=n^2(g-1)+1.$$
See \cite{hit87} for details. By projectivizing the canonical bundle, i.e.,  $\mathbb{P}(\K_\mathfrak{R}\oplus\mathbb{C})$, we still conclude that a Hitchin spectral curve $\widehat{\mathfrak{R}}$ in $\K_\mathfrak{R}$ naturally sits in $\mathfrak{S}'=\mathbb{P}(\K_\mathfrak{R}\oplus\mathbb{C})$. However, it is easy to see that $\widehat{\mathfrak{R}}$ is not linearly equivalent to  a divisor $nK_{\mathfrak{R}}f'$ where $qf'=\pi_{\mathfrak{S}'}^{-1}(q)$ and $\pi_{\mathfrak{S}'}:\mathfrak{S}'\to\mathfrak{R}$, since there is no tautological section on $\mathfrak{S}'$. Instead, it is not difficult to see that  $\widehat{\mathfrak{R}}$ is linearly equivalent to $nC'_0+nK_{\mathfrak{R}}f'$ where $C'_0$ is the section corresponding to a surjection $\mathcal{K}_\mathfrak{R}\oplus\mathcal{O}_\mathfrak{R}\to\mathcal{O}_\mathfrak{R}$ where $\mathcal{K}_\mathfrak{R}$ is the sheaf of a canonical bundle $\K_\mathfrak{R}$.  Let us calculate the dimension of $|nC'_0+nK_\mathfrak{R}f'|$. In order to do that, we need the following Lemma.

\begin{lemma}\label{lemma1}
For integers $i\geq0$ and $n_1,n_2>0$, we have
\begin{equation}\label{indu1}
\HH^i(\mathcal{O}_{\mathfrak{S}'}(n_1C'_0+n_2K_\mathfrak{R}f'))=\sum_{m=0}^{n_1}\HH^i(\mathcal{O}_{\mathfrak{R}}(m\mathfrak{e}+n_2K_\mathfrak{R})).
\end{equation}
Here $\mathfrak{e}=\bigwedge^2\mathcal{E}_0$ where $\mathfrak{S}'=\mathbb{P}(\mathcal{E}_0)$ and $\mathcal{E}_0$ is normalized.
\end{lemma}

\begin{proof} 
Let $\mathfrak{S}'=\mathbb{P}(\mathcal{E}_0)$ where $\mathcal{E}_0=\mathcal{O}_\mathfrak{R}\oplus(-\mathcal{K}_\mathfrak{R})$, $\pi_{\mathfrak{S}'}:\mathfrak{S}'\to\mathfrak{R}$, and $\mathcal{K}_\mathfrak{R}$ is the sheaf of a canonical bundle $\K_\mathfrak{R}$. . Note that $\mathbb{P}(\mathcal{K}_\mathfrak{R}\oplus\mathcal{O}_\mathfrak{R})\cong\mathbb{P}(\mathcal{O}_\mathfrak{R}\oplus(-\mathcal{K}_\mathfrak{R}))$. Since $\mathcal{O}_\mathfrak{R}\oplus(-\mathcal{K}_\mathfrak{R})$ is normalized (see \emph{p.374} in \cite{har77}), the divisor of $\bigwedge^2\mathcal{E}_0$ is $-K_\mathfrak{R}:=\mathfrak{e}$. We will prove the lemma by proceeding an induction on $n_1$. From \emph{p.371} in \cite{har77}, we know that 
$$\HH^i(\mathcal{O}_{\mathfrak{S}'}(C'_0+n_2K_\mathfrak{R}f'))=\HH^i(\mathfrak{R},\pi_{\mathfrak{S}'\ast}(\mathcal{O}_{\mathfrak{S}'}(C'_0+n_2K_\mathfrak{R}f'))=\HH^i(\mathfrak{R},\mathcal{E}_0\otimes (n_2\mathcal{K}_\mathfrak{R})).$$
Consequently,
$$\HH^i(\mathcal{O}_{\mathfrak{S}'}(C'_0+n_2K_\mathfrak{R}f'))=\HH^i(\mathcal{O}_\mathfrak{R}(n_2K_\mathfrak{R}))\oplus\HH^i(\mathfrak{R},\mathcal{O}_\mathfrak{R}((n_2-1)K_\mathfrak{R})).$$
Hence, we have proved assertion~\eqref{indu1} for $n_1=1$. Let
$$D=n_1C'_0+n_2K_\mathfrak{R}f'\text{ and }D_{-1}=(n_1-1)C'_0+(n_2-1)K_\mathfrak{R}f'.$$
Now consider a short exact sequence:
$$0\to\mathcal{O}_{\mathfrak{S}'}(D_{-1})\to\mathcal{O}_{\mathfrak{S}'}(D)\to\mathcal{O}_{C'_0+K_\mathfrak{R}f'}(D)\to0.$$
Hence, we have the induced long exact sequence:
$$\xymatrix{0\ar[r]&\HH^0(\mathcal{O}_{\mathfrak{S}'}(D_{-1}))\ar[r]&\HH^0(\mathcal{O}_{\mathfrak{S}'}(D))\ar[r]^{\alpha}&\HH^0(\mathfrak{R},n_2\mathcal{K}_\mathfrak{R})\ar[dll]_{\beta_0}&\\
&\HH^1(\mathcal{O}_{\mathfrak{S}'}(D_{-1}))\ar[r]&\HH^1(\mathcal{O}_{\mathfrak{S}'}(D))\ar[r]&\HH^1(\mathfrak{R},n_2\mathcal{K}_\mathfrak{R})\ar[dll]_{\beta_1}&\\
&\HH^2(\mathcal{O}_{\mathfrak{S}'}(D_{-1}))\ar[r]&\HH^2(\mathcal{O}_{\mathfrak{S}'}(D))\ar[r]&0.&}$$
Since $\alpha$ is always surjective, $\beta_0$ is a zero map. By the induction hypothesis, we have 
$$\HH^2(\mathcal{O}_{\mathfrak{S}'}(D_{-1}))=\sum_{m=0}^{n_1-1}\HH^2(\mathcal{O}_{\mathfrak{R}}(m\mathfrak{e}+(n_2-1)K_\mathfrak{R})).$$
Clearly, the right-hand side is zero. Hence, $\beta_1$ is a zero map. Consequently, we have
\begin{equation}\label{indu2}
\HH^i(\mathcal{O}_{\mathfrak{S}'}(D))=\HH^i(\mathcal{O}_{\mathfrak{S}'}(D_{-1}))\oplus\HH^i(\mathfrak{R},n_2\mathcal{K}_\mathfrak{R})\text{ for }i\geq0.
\end{equation}
Again, by the induction hypothesis, we have
$$\HH^i(\mathcal{O}_{\mathfrak{S}'}(D_{-1}))=\sum_{m=0}^{n_1-1}\HH^i(\mathcal{O}_{\mathfrak{R}}(m\mathfrak{e}+(n_2-1)K_\mathfrak{R})).$$
Since $\mathfrak{e}=-K_\mathfrak{R}$, we see that
$$\begin{aligned}
\sum_{m=0}^{n_1-1}\HH^i(\mathcal{O}_{\mathfrak{R}}(m\mathfrak{e}+(n_2-1)K_\mathfrak{R}))&=\sum_{m=0}^{n_1-1}\HH^1(\mathcal{O}_{\mathfrak{R}}(m+1)\mathfrak{e}+n_2K_\mathfrak{R}))\\
&=\sum_{m=1}^{n_1}\HH^i(\mathcal{O}_{\mathfrak{R}}(m\mathfrak{e}+n_2K_\mathfrak{R})).
\end{aligned}$$
Combing this with \eqref{indu2}, we have
$$\HH^i(\mathcal{O}_{\mathfrak{S}'}(D))=\sum_{m=0}^{n_1}\HH^i(\mathcal{O}_{\mathfrak{R}}(m\mathfrak{e}+n_2K_\mathfrak{R})).$$
\end{proof}

\begin{corollary}\label{codim1}
$$\begin{aligned}
\dim\HH^1(\mathcal{O}_{\mathfrak{S}'}(nC'_0+nK_\mathfrak{R}f'))&=g+1\\
\dim\HH^2(\mathcal{O}_{\mathfrak{S}'}(nC'_0+nK_\mathfrak{R}f'))&=0.
\end{aligned}$$
\end{corollary}

\begin{proof}
By Lemma~\ref{lemma1} and $\dim\HH^1(\mathcal{O}_{\mathfrak{R}}((n-m)K_\mathfrak{R}))=0$ for $n-m>1$, we have
$$\dim\HH^1(\mathcal{O}_{\mathfrak{S}'}(nC'_0+nK_\mathfrak{R}f'))=\dim\HH^1(\mathcal{O}_{\mathfrak{R}})+\dim\HH^1(\mathcal{O}_{\mathfrak{R}}(K_\mathfrak{R}))=g+1.$$
The second assertion also comes from Lemma~\ref{lemma1}:
$$\HH^2(\mathcal{O}_{\mathfrak{S}'}(D))=\sum_{m=0}^{n_1}\HH^2(\mathcal{O}_{\mathfrak{R}}(m\mathfrak{e}+n_2K_\mathfrak{R}))=0.$$
\end{proof}

Consequently, we have the following result:

\begin{theorem}\label{mainth2}
For $\widehat{\mathfrak{R}}\in|nC'_0+nK_\mathfrak{R}f'|$ on $\mathfrak{S}'$, the genus of  $\widehat{\mathfrak{R}}$ is given by
$$g(\widehat{\mathfrak{R}})=n^2(g-1)+1.$$
Moreover,
$$\dim|nC'_0+nK_\mathfrak{R}f'|=n^2(g-1)+1.$$
\end{theorem}

\begin{proof}
Note that
$$\K_{\mathfrak{S}'}=-2C'_0\text{ and }C'_0.C'_0=2-2g.$$
Consequently, the adjunction formula implies
$$\begin{aligned}
g(\widehat{\mathfrak{R}})&=\frac{\Big(nC'_0+nK_\mathfrak{R}f'\Big).\Big((n-2)C'_0+nK_\mathfrak{R}f'\Big)}{2}+1\\
&=-n(n-2)(g-1)+n^2(g-1)+n(n-2)(g-1)+1\\
&=n^2(g-1)+1.
\end{aligned}$$
Let $D=nC'_0+nK_\mathfrak{R}f'$. By Corollary~\ref{codim1}, we have  $\dim\HH^2(\mathcal{O}_{\mathfrak{S}'}(D))=0$ and $\dim\HH^1(\mathcal{O}_{\mathfrak{S}'}(D))=g+1$.  Hence, by the Riemann-Roch theorem, we have
$$\begin{aligned}
\dim\HH^0(\mathcal{O}_{\mathfrak{S}'}(D))&=\frac{D.(D-K_{\mathfrak{S}'})}{2}+(1-g)+\dim\HH^1(\mathcal{O}_{\mathfrak{S}'}(D))\\
&=(n^2-1)(g-1)+g+1\\
&=n^2(g-1)+2.
\end{aligned}$$
\end{proof}

Theorem~\ref{mainth2} implies that there is no difference of natural properties between the linear system $|n\mathfrak{R}|$ of Hitchin divisors on the non-compact space $\K_\mathfrak{R}$ and the linear system $|nC'_0+nK_\mathfrak{R}f'|$ of Hitchin divisors on the compactified space $\mathfrak{S}'=\mathbb{P}(\K_\mathfrak{R}\oplus\mathbb{C})$. In the next section, we will see what happens if we replace $\mathfrak{S}'$ with $\mathfrak{S}$.

\section{Hitchin covers and the ruled surface $\mathfrak{S}$.}\label{sechit}

In this section, we will characterize the Hitchin divisors mapped into the surface $\mathfrak{S}$ constructed in Section~\ref{sec;uni} and the singularities of the Hitchin divisors in the associated linear system on $\mathfrak{S}$. Let us remind
$$\mathfrak{S}=\elm_{p'_1}\circ\dots\circ \elm_{p'_{2g-2}}\mathfrak{S}'$$
where $p'_i=([-1,1],q_i)$ for $i=1,\dots,2g-2$ and $K_\mathfrak{R}=\sum_{i=1}^{2g-2}q_i$ is a canonical divisor of $\mathfrak{R}$. By abuse of notation, let $\elm$ be the transformation from $\mathfrak{S}'$ to $\mathfrak{S}$:
$$\elm:\mathfrak{S}'\to\elm_{p'_1}\circ\dots\circ \elm_{p'_{2g-2}}\mathfrak{S}'.$$
Clearly,
$$\elm(nC'_0+nK_\mathfrak{R}f')=nC_0+nK_\mathfrak{R}f.$$

\begin{lemma}\label{mainth3}
The genus of a curve $D\in|nC_0+nK_\mathfrak{R}f|$ is $(2n^2-n)(g-1)+1$.
\end{lemma}

\begin{proof}
Note that $C_0.C_0=qf.qf=0$ and $C_0.qf=1$ for any $q\in\mathfrak{R}$. Since $\bigwedge^2\WW$ is trivial, we see that
$$K_\mathfrak{S}=-2C_0+K_\mathfrak{R}f.$$
Using the adjunction formula, we have
$$\begin{aligned}
2\widehat{g}-2&=(K_\mathfrak{S}+D).D\\
&=(-2C_0+K_\mathfrak{R}f+nC_0+nK_\mathfrak{R}f).(nC_0+nK_\mathfrak{R}f)\\
&=((n-2)C_0+(n+1)K_\mathfrak{R}f).(nC_0+nK_\mathfrak{R}f)\\
&=n(n-2)(2g-2)+n(n+1)(2g-2).
\end{aligned}$$
Hence $\widehat{g}=(2n^2-n)(g-1)+1$.
\end{proof}

\begin{definition}\label{hitandef1}
We will call a Hitchin spectral cover $\pi:\widehat{\mathfrak{R}}\to \mathfrak{R}$ of degree $n$ a \emph{Hitchin tangential cover} on $\mathfrak{S}$ if there is a morphism $\iota:\widehat{\mathfrak{R}}\to \mathfrak{S}:=\mathbb{P}(\WW)$ such that $\iota^{-1}(C_0)\subseteq\pi^{-1}(K_\mathfrak{R})$.
\end{definition}

\begin{lemma}\label{lemmatan}
If $\pi:\widehat{\mathfrak{R}}\to \mathfrak{R}$ is a Hitchin tangential cover, then
$$\iota(\widehat{\mathfrak{R}})\in |nC_0+nK_\mathfrak{R}f|.$$ 
\end{lemma}

\begin{proof}
Let $i:\K_\mathfrak{R}\to\mathbb{P}(\K_\mathfrak{R}\oplus\mathbb{C})$ be a natural inclusion. Note that by construction, a Hitchin spectral curve is defined by the zero locus of a polynomial of degree $n$
$$\widehat{\mathfrak{R}}=\{\det(\chi\cdot \id-\phi(z))=0 \}:=\{\mathcal{P}_{\widehat{\mathfrak{R}}}(\chi,z)=0\}$$
where $\chi$ is a tautological section of $\pi^{\ast}_{\K_\mathfrak{R}}\K_\mathfrak{R}$ and $\id$ is the identity map in $\End\E$. This polynomial is of degree $n$ in $\chi$. If $\pi:\widehat{\mathfrak{R}}\to \mathfrak{R}$ is a Hitchin tangential cover, we have the following commutative diagram of morphisms,
$$\xymatrix{
\widehat{\mathfrak{R}}\ar[r]^{\iota}\ar[dr]_{i}&\mathfrak{S}\\
&\mathfrak{S}':=\mathbb{P}(\K_\mathfrak{R}\oplus\mathbb{C}).\ar[u]_{\elm:=\elm_{p'_1}\circ\dots\circ \elm_{p'_{2g-2}}}}$$
In particular, we may find a polynomial $P_{\iota(\widehat{\mathfrak{R}})}(k_\mathfrak{S},z)$ on $\mathfrak{S}$ such that its zero is $\iota(\widehat{\mathfrak{R}})$ and the pole is $nC_0+nK_\mathfrak{R}f$ where 
$$k_\mathfrak{S}=T\circ \Big(\elm_{p_1}\circ\dots\circ \elm_{p_{2g-2}}\circ\elm_{p'_1}\circ\dots\circ \elm_{p'_{2g-2}}\Big)^{-1}$$
and $T:\mathfrak{R}\times\mathbb{P}^1\to\mathbb{P}^1$ is a natural projection. 
\end{proof}

\begin{definition}
We will call a meromorphic function $k:=k_\mathfrak{S}\circ \iota$ on $\widehat{\mathfrak{R}}$ a \emph{Hitchin tangential function}. Moreover, we will denote $\mathscr{HT}(n,g,\mathfrak{S})$ is the sub-linear system of $|nC_0+nK_\mathfrak{R}f|$ consisting of Hitchin tangential covers where $g$ is the genus of $\mathfrak{R}$.
\end{definition}

From Lemma~\ref{mainth3}, we know that the arithmetic genus of $\iota(\widehat{\mathfrak{R}})$ is $(2n^2-n)(g-1)+1$ which is bigger than $n^2(g-1)+1$, the genus of $\widehat{\mathfrak{R}}$. This implies that  $\iota(\widehat{\mathfrak{R}})$ should admit singularities. On the other hands, it is not hard to see that $\dim|nC_0+nK_\mathfrak{R}f|$ is also bigger than the desired one, $n^2(g-1)+1$. We will characterize a sub-linear system of $|nC_0+nK_\mathfrak{R}f|$ whose members have the genus $n^2(g-1)+1$ after their desingularization.

\begin{remark}\label{remark1}
We remark the followings:
\begin{itemize}
\item[(i)]
The pole of $k_\mathfrak{S}$ is $C_0+K_\mathfrak{R}f$ and $P_{\iota(\widehat{\mathfrak{R}})}(k_\mathfrak{S},z)$ is a polynomial of degree $n$ in $k_\mathfrak{S}$.
\item[(ii)]
The reason for choosing the terminology, ``tangential cover'',  in Definition~\ref{hitandef1} will be clearer in Section~\ref{secpol}.
\item[(iii)]
Later we will show the existence of a Hitchin tangential cover by calculating the dimension of $\mathscr{HT}(n,g,\mathfrak{S})$ in Corollary~\ref{dimco7}.
\end{itemize}
\end{remark}

\begin{lemma}\label{hitth1}
Let $q\in\mathfrak{R}$. For any extension class $[\alpha]\in\HH^1(\mathfrak{R},\mathcal{O}_{\mathfrak{R}})$, we have
$$[\alpha]\in\delta\Big(\HH^0(\mathfrak{R},\mathbb{C}_{K_\mathfrak{R}+q})\Big)=\HH^1(\mathfrak{R},\mathcal{O}_{\mathfrak{R}}).$$
\end{lemma}

\begin{proof}
Consider the following sequence
$$0\to\mathcal{O}_{\mathfrak{R}}\overset{s}{\to}\mathcal{O}_\mathfrak{R}(K_\mathfrak{R}+q)\to\mathbb{C}_{K_\mathfrak{R}+q}\to0.$$
From this, we have a long exact sequence
$$\xymatrix{\cdots\ar[r]&\HH^0(\mathfrak{R},\mathbb{C}_{K_\mathfrak{R}+q})\ar[r]^{\delta}\ar@{=}[d]&\HH^1(\mathfrak{R},\mathcal{O}_\mathfrak{R})\ar[r]^(0.4){s^\ast}\ar@{=}[d]&\HH^1(\mathfrak{R},\mathcal{O}_\mathfrak{R}(K_\mathfrak{R}+q))\ar@{=}[d]\\
&\mathbb{C}^{2g-1}&\mathbb{C}^g&0.}$$
So we know that
$$s^\ast([\alpha])=0\in\HH^0(\mathfrak{R},\mathcal{O}_\mathfrak{R}(K_\mathfrak{R}+q)).$$
Hence, from the exactness, we have
$$[\alpha]\in\delta\Big(\HH^0(\mathfrak{R},\mathbb{C}_{K_\mathfrak{R}+q})\Big)\subseteq \HH^1(\mathfrak{R},\mathcal{O}_{\mathfrak{R}}).$$
\end{proof}

\begin{corollary}\label{co61}
Let $\pi:\widehat{\mathfrak{R}}\to\mathfrak{R}$ be any ramified cover. Then
$$\pi^\ast([\alpha])\in\delta\Big(\HH^0(\widehat{\mathfrak{R}},\mathbb{C}_{\pi^{-1}(K_\mathfrak{R}+q)}\Big)
\subset \HH^1(\widehat{\mathfrak{R}},\mathcal{O}_{\widehat{\mathfrak{R}}}).$$
\end{corollary}

Now consider the following diagram

$$\xymatrix{
0\ar[r]&\mathcal{O}_\mathfrak{R}\ar[r]\ar@{=}[d]&\mathcal{O}_\mathfrak{R}(K_\mathfrak{R})\ar[r]\ar[d] &\mathbb{C}_{K_\mathfrak{R}}\ar[d]\ar[r]&0\\
0\ar[r]&\mathcal{O}_\mathfrak{R}\ar[r]&\mathcal{O}_\mathfrak{R}(K_\mathfrak{R}+q)\ar[r]
&\mathbb{C}_{K_\mathfrak{R}+q}\ar[r]& 0\\
0\ar[r]&\mathcal{O}_\mathfrak{R}\ar[r]\ar@{=}[u]
&\mathcal{O}_\mathfrak{R}(q)\ar[r]\ar[u]&\mathbb{C}_{q}\ar[r]\ar[u] &0.  }$$
This induces

$$\xymatrix{
\HH^0(\mathcal{O}_\mathfrak{R}(K_\mathfrak{R}))\ar[r]\ar[d] &\HH^0(\mathbb{C}_{K_\mathfrak{R}})\ar[r]^{\delta_{K_\mathfrak{R}}}\ar[d] &\HH^1(\mathcal{O}_\mathfrak{R})\ar[r]\ar@{=}[d]         &\mathbb{C}^1\ar[r]&0\\
\HH^0(\mathcal{O}_\mathfrak{R}(K_\mathfrak{R}+q))\ar[r]
&\HH^0(\mathbb{C}_{K_\mathfrak{R}+q})\ar[r]^{\delta_{K_\mathfrak{R}+q}}&   \HH^1(\mathcal{O}_{\mathfrak{R}})\ar[r] &0\ar[r]&0  \\
\HH^0(\mathcal{O}_\mathfrak{R}(q))\ar[r]^{\text{zero map}}\ar[u]
&\HH^0(\mathbb{C}_q)\ar[r]^{\delta_q}\ar[u]&   \HH^1(\mathcal{O}_{\mathfrak{R}})\ar@{=}[u]\ar[r] &\mathbb{C}^{g-1}\ar[r]&0.  }$$
So we see that
\begin{equation}\label{sp1}
\delta_{K_\mathfrak{R}+q}=\delta_{K_\mathfrak{R}}+\delta_{q}.
\end{equation}

\begin{corollary}\label{coclas}
There is a class $[\nu]\in\HH^0(\mathbb{C}_{K_\mathfrak{R}})$ such that
$$[\alpha]=\delta_{K_\mathfrak{R}}([\nu])+c\delta_{q}(1_{\mathbb{C}_q})$$
where $1_{\mathbb{C}_q}$ is a generator of $\HH^0(\mathbb{C}_q)\cong\mathbb{C}^1$ and $c$ is a constant.
\end{corollary}

We can improve Corollary~\ref{co61} for a case when $[\alpha]$ is the class of the constructed extension in Section~\ref{sec;uni}
and $\pi:\widehat{\mathfrak{R}}\to\mathfrak{R}$ is a Hitchin tangential cover.

\begin{theorem}\label{thhi1}
Let $[\alpha]\in\HH^1(\mathfrak{R},\mathcal{O}_\mathfrak{R})$ be the class of the constructed extension in Section~\ref{sec;uni}
$$0\to\mathcal{O}_\mathfrak{R}\to\WW\to\mathcal{O}_\mathfrak{R}\to0.$$
Let $\pi:\widehat{\mathfrak{R}}\to\mathfrak{R}$ be a Hitchin tangential cover. Then
$$\pi^\ast([\alpha])\in\delta\Big(\HH^0(\widehat{\mathfrak{R}},\mathbb{C}_{\pi^{-1}(K_\mathfrak{R})}\Big)
\subset \HH^1(\widehat{\mathfrak{R}},\mathcal{O}_{\widehat{\mathfrak{R}}}).$$
\end{theorem}

\begin{proof}
From Definition~\ref{hitandef1}, there is a morphism $\iota:\widehat{\mathfrak{R}}\to \mathfrak{S}:=\mathbb{P}(\WW)$ such that $\iota^{-1}(C_0)\subseteq\pi^{-1}(K_\mathfrak{R})$. Now it is well-known that the existence of a morphism $\iota:\widehat{\mathfrak{R}}\to \mathfrak{S}:=\mathbb{P}(\WW)$ such that $\iota^{-1}(C_0)\subseteq\pi^{-1}(K_\mathfrak{R})$
is equivalent to the existence of a surjective homomorphism $c$ in the following diagram (see \emph{p.162 in \cite{har77}}):
$$\xymatrix{&0\ar[d]&&&\\
0\ar[r] &\mathcal{O}_{\widehat{\mathfrak{R}}}\ar[r]\ar[d]^{s} &\pi^\ast\mathcal{W}\ar@{.>}[d]^{s'}\ar[r]^b\ar@{.>}[dl]|-{c}       &\mathcal{O}_{\widehat{\mathfrak{R}}}\ar@{=}[d]\ar[r]  &0\\
0\ar[r] &\mathcal{O}_{\widehat{\mathfrak{R}}}(\pi^{-1}(K_\mathfrak{R}))\ar[r]\ar[d] &\mathcal{W}'\ar[r]         &\mathcal{O}_{\widehat{\mathfrak{R}}}\ar[r]  &0\\
&\mathbb{C}_{\pi^{-1}(K_\mathfrak{R})}\ar[d] &         &&\\
&0&&&}$$
Note that the existence of $c$ implies the existence of $s'$, i.e., the commutativity of the diagram where the second horizontal short exact sequence is assumed to split.
So $s^\ast(\pi^\ast([\alpha]))=0$ in $\HH^1(\mathcal{O}_{\widehat{\mathfrak{R}}}(\pi^{-1}(K_\mathfrak{R})))$. Consequently, we deduce that
$$\pi^\ast([\alpha])\in\delta\Big(\HH^0(\widehat{\mathfrak{R}},\mathbb{C}_{\pi^{-1}(K_\mathfrak{R})}\Big)
\subset \HH^1(\widehat{\mathfrak{R}},\mathcal{O}_{\widehat{\mathfrak{R}}}).$$
\end{proof}

Immediate consequences are the following corollaries:

\begin{corollary}\label{corotan1}
For the class $[\alpha]$ of the constructed extension in Section~\ref{sec;uni}, we have
$$[\alpha]\in\delta\Big(\HH^0(\mathfrak{R},\mathbb{C}_{K_\mathfrak{R}})\Big)\subseteq \HH^1(\mathfrak{R},\mathcal{O}_{\mathfrak{R}}).$$
\end{corollary}

\begin{corollary}\label{corotan2}
For a given Hitchin tangential cover $\pi:\widehat{\mathfrak{R}}\to\mathfrak{R}$ and the constructed extension $[\alpha]$ in Section~\ref{sec;uni} , there exists $\rho_\pi\in\HH^0(\widehat{\mathfrak{R}},\mathbb{C}_{\pi^{-1}(K_\mathfrak{R})})$ such that
\begin{equation}\label{tanhit}
\delta\rho_\pi=\pi^\ast([\alpha]).
\end{equation}

\end{corollary}

The characterization of  $\iota(\widehat{\mathfrak{R}})$ in $|nC_0+nK_\mathfrak{R}f|$ among divisors $D\in|nC_0+nK_\mathfrak{R}f|$ is, by the construction, that $\iota(\widehat{\mathfrak{R}})$ passes through each point of $\sigma(K_\mathfrak{R})$ with multiplicity $n$ for  $\widehat{\mathfrak{R}}\in|nC'_0+nK_\mathfrak{R}f'|$ where $\sigma:\mathfrak{R}\to C_0$ is the section. Hence $\iota(\widehat{\mathfrak{R}})$ becomes a singular curve in $\mathfrak{S}$. In particular \emph{the degree of singularities} $\delta_{\sigma(q_i)}$ is $\frac{n(n-1)}{2}$. If the singularity of multiplicity $n$ at each point in $\sigma(K_\mathfrak{R})$ is an ordinary singularity of multiplicity $n$, then the desingularization of the singularities is nothing but the blowing up of $\mathfrak{S}$ at $\sigma(K_\mathfrak{R})$. Let $\bl:\widetilde{\mathfrak{S}}\to\mathfrak{S}$ be the blow-up of $\mathfrak{S}$ at $\sigma(K_\mathfrak{R})$. Let $\widetilde{\iota(\widehat{\mathfrak{R}})}$ is the strict transformation of $\iota(\widehat{\mathfrak{R}})$ with respect to $\bl$. Then the genus $\widetilde{\widehat{g}}$ of $\widetilde{\iota(\widehat{\mathfrak{R}})}$ is given by
$$\widetilde{\widehat{g}}=(2n^2-n)(g-1)+1-\frac{n(n-1)}{2}(2g-2)=n^2(g-1)+1.$$
Let
$$L_{n,n,n}:=\bl^\ast(nC_0+nK_\mathfrak{R}f)-n\sum_{i=1}^{2g-2}E_i.$$
Clearly, $\widetilde{\iota(\widehat{\mathfrak{R}})}\in|L_{n,n,n}|$. By the Bertini theorem, a generic divisor in $|L_{n,n,n}|$ is smooth. Hence, we conclude that $\iota(\widehat{\mathfrak{R}})$ has an ordinary singularity of multiplicity $n$ for a generic Hitchin spectral curve $\pi:\widehat{\mathfrak{R}}\to \mathfrak{R}$.

One can observe that the above would be the typical procedure: Take a non-compact complex surface and consider a linear system of smooth divisors. If one wants to study a moduli space consisting of divisors with a particular property, then construct another compact complex surface where the singular divisors naturally form a sub-linear system of some linear system.

\begin{lemma}\label{lemma2}
For $n>2$, we have
$$\begin{aligned}
\dim\HH^1(\mathcal{O}_{\widetilde{\mathfrak{S}}}(L_{n,n,n}))&=0\\
\dim\HH^2(\mathcal{O}_{\widetilde{\mathfrak{S}}}(L_{n,n,n}))&=0.
\end{aligned}$$
\end{lemma}

\begin{proof}
Note the facts that $E_i.E_j=-\delta_{ij},\bl^\ast(D_1).E_i=0$, and $\bl^\ast(D_1).\bl^\ast(D_2)=D_1.D_2$ for $D_1,D_2\in\pic(\mathfrak{S})$ where $\delta_{ij}$ is the Kronecker symbol. Let 
$$D=\bl^\ast((n+2)C_0+(n-1)K_\mathfrak{R})-(n+1)\sum_{i=1}^{2g-2}E_i.$$
Clearly, $D^2=(n^2-5)(2g-2)>0$ for $n>2$. Also $D.C_0>0$, $D.qf>0$, and $D.E_i>0$ for the generators of $\pic(\widetilde{\mathfrak{S}})$ where $q\in\mathfrak{R}$. Hence from the Nakai-Moishezon criterion (\emph{p.365} in \cite{har77}), we see that $D$ is an ample divisor. Note that the canonical divisor of $\widetilde{\mathfrak{S}}$ is
$$\K_ {\widetilde{\mathfrak{S}}}=\bl^\ast(\K_\mathfrak{S})+\sum_{i=1}^{2g-2}E_i=\bl^\ast(-2C_0+K_\mathfrak{R}f)+\sum_{i=1}^{2g-2}E_i.$$
By the Kodaira vanishing theorem (\emph{p.248} in \cite{har77}), we have for $i>0$,
$$\begin{aligned}
0&=\HH^i(\mathcal{O}_{\widetilde{\mathfrak{S}}}(D+K_{\widetilde{\mathfrak{S}}}))\\
&=\HH^i(\mathcal{O}_{\widetilde{\mathfrak{S}}}(\bl^\ast(nC_0+nK_\mathfrak{R})-n\sum_{i=1}^{2g-2}E_i))\\
&=\HH^i(\mathcal{O}_{\widetilde{\mathfrak{S}}}(L_{n,n,n})).
\end{aligned}$$
\end{proof}

From the above lemma, we have the following theorem.

\begin{theorem}\label{thdim}
Let $n>2$. The dimension of a linear system $|L_{n,n,n}|$ on $\widetilde{\mathfrak{S}}$ is
$$\dim|L_{n,n,n}|=(n^2-1)(g-1)-1.$$
\end{theorem}

\begin{proof}
Note that since $p_a$ is preserved under a blow-up, we have
$$p_a(\mathcal{O}_{\widetilde{\mathfrak{S}}})=p_a(\mathcal{O}_{\mathfrak{S}})=-g.$$
By Lemma~\ref{lemma2}, $\dim\HH^2(\mathcal{O}_{\widetilde{\mathfrak{S}}}(L_{n,n,n}))=0$ and $\dim\HH^1(\mathcal{O}_{\widetilde{\mathfrak{S}}}(L_{n,n,n}))=0$. By the Riemann-Roch theorem, we see that
$$\begin{aligned}
\dim\HH^0(\mathcal{O}_{\widetilde{\mathfrak{S}}}(L_{n,n,n}))&=\frac{L_{n,n,n}.(L_{n,n,n}-K_{\widetilde{\mathfrak{S}}})}{2}+(1-g)\\
&=\Big(n(n-1)+n(n+2)-n(n+1)-1\Big)(g-1)\\
&=(n^2-1)(g-1).
\end{aligned}$$

\end{proof}

\begin{corollary}\label{dimco7}
For $n>2$, $\dim\mathscr{HT}(n,g,\mathfrak{S})=(n^2-1)(g-1)-1$.
\end{corollary}

\section{Hitchin tangential functions and the generalization of tangency}\label{secpol}

From the definition of the skyscraper sheaf $\mathbb{C}_D$ where $D$ is a divisor on $\mathfrak{R}$, we may regard an element of $\HH^0(\mathfrak{R},\mathbb{C}_{D})$ as the Laurent tail of a local function defined on the support of $D$ on $\mathfrak{R}$. That is, there is a correspondence between  a vector $\big((c_{1,1},\dots,c_{1,m_1}),\dots,(c_{d,1},\dots,c_{d,m_d})\big)\in\mathbb{C}^{\sum_{k=1}^{d}m_k}\cong\HH^0(\mathfrak{R},\mathbb{C}_D)$ where $D=\sum_{k=1}^{d}m_kp_k$ and a local function $\lambda$ with a Laurent tail
$$\lambda(z_k)=\frac{c_{k,m_k}}{z_{k}^{m_k}}+\cdots+\frac{c_{k,1}}{z_{k}}$$
in the neighborhood of $p_k$ in the support of $D$. In this vein, we can formulate the following:

\begin{definition}
We call the set of the Laurent tails of a local function $\lambda$ at a divisor $D$ a \emph{residue section} of $\lambda$ associated with a divisor $D$ and denote it by
$$\rho(\lambda)\in \HH^0(\mathfrak{R},\mathbb{C}_{D}).$$
Conversely, we assign a class $\rho\in \HH^0(\mathfrak{R},\mathbb{C}_{D})$ a local function denoted by $\lambda_\rho$. In particular, we have
$$\rho(\lambda_\rho)=\rho.$$
\end{definition}

By Corollary~\ref{corotan1}, we know that there is $\rho\in \HH^0(\mathfrak{R},\mathbb{C}_{K_\mathfrak{R}})$ such that
$$\pi^\ast\Big(\delta(\rho)\Big)=\pi^\ast([\alpha])\in\delta\Big(\HH^0(\widehat{\mathfrak{R}},\mathbb{C}_{\pi^{-1}(K_\mathfrak{R})})\Big)
\subset \HH^1(\widehat{\mathfrak{R}},\mathcal{O}_{\widehat{\mathfrak{R}}}).$$
From now on we let $\lambda_\rho$ be a local function corresponding to $\rho$.

\begin{theorem}\label{thtan1}
Let $\pi:\widehat{\mathfrak{R}}\to \mathfrak{R}$ be a Hitchin tangential cover of degree $n$.
There is a meromorphic function $k$ on $\widehat{\mathfrak{R}}$ such that the poles of $k+\pi^\ast(\lambda_\rho)$ is $\iota^{-1}(C_0)\subseteq\pi^{-1}(K_\mathfrak{R})$.
\end{theorem}

\begin{proof}
We know that there is a meromorphic function $T:\mathfrak{R}\times\mathbb{P}^1\to\mathbb{P}^1$ which is a natural projection. Then we may let
$$k_\mathfrak{S}=T\circ \Big(\elm_{p_1}\circ\dots\circ \elm_{p_{2g-2}}\circ\elm_{p'_1}\circ\dots\circ \elm_{p'_{2g-2}}\Big)^{-1}.$$
Consider a Hitchin tangential function 
$$k=k_\mathfrak{S}\circ \iota\text{ where }\iota:\widehat{\mathfrak{R}}\to \mathfrak{S}.$$
The poles of $k_\mathfrak{S}$ are $C_0+K_\mathfrak{R}f$. Hence the poles of $k_\mathfrak{S}+\pi^{\ast}_{\mathfrak{S}}(\lambda_\rho)$ is holomorphic outside of $C_0$ where 
$\pi_\mathfrak{S}:\mathfrak{S}\to\mathfrak{R}$. Consequently,
$$\iota^\ast\Big(k_\mathfrak{S}+\pi^{\ast}_{\mathfrak{S}}(\lambda_\rho)\Big)=k+\pi^{\ast}(\lambda_\rho)$$
is holomorphic outside of $C_0$.
\end{proof}

In the theory of the K-P equation, the \textit{Its-Mateev formula} implies that
$$u(x,y,t)=2\frac{\partial^2}{\partial x^2}\ln \theta(\mathbf{U}x+\mathbf{V}y+\mathbf{W}t+z_0)+\text{constant}$$
and the tangent vector of the first K-P flow is $\mathbf{U}$, which is nothing but $\dAb(\TT_p)$ where $\TT_p$ is a tangent vector at $p$ of a Riemann surface $\widehat{\mathfrak{R}}$ (see \emph{p.287} in \cite{kr80}). In \cite{tv90}, $\widehat{\mathfrak{R}}$ is called a {\em tangential cover} of an elliptic curve $X$ if $\dAb(\TT_p)=\pi^\ast(\dAb(\TT_q))$ where $\TT_q$ is a tangent vector at $q$ and $\pi(p)=q$ and Treibich and Verdier proved that any elliptic soliton is a tangential cover.
Note that $\delta(1_{\mathbb{C}_p})=\dAb(\TT_p)=\mathbf{U} \in\delta\Big(\HH^0(\widehat{\mathfrak{R}},\mathbb{C}_{p})\Big)\subset \HH^1(\widehat{\mathfrak{R}},\mathcal{O}_{\widehat{\mathfrak{R}}})$ and since $\HH^1(X,\mathcal{O}_X)$ is $1$-dimensional, in this case $[\alpha]$ is realized as the tangent vector $\TT_q$, i.e., $\dAb(\TT_q)=[\alpha]$ where $[\alpha]\in\HH^1(X,\mathcal{O}_X)\cong\mathbb{C}^1$ is the canonical extension class
$$0\to\mathcal{O}_X\to\mathcal{W}\to\mathcal{O}_X\to0.$$
Consequently, what the tangency condition for an elliptic curve case implies is that the tangency condition of the K-P flow is realized by the lifting $\pi^\ast([\alpha])$ of some ruled surface corresponding to $[\alpha]$: Algebraically, we may write this as
\begin{equation}\label{ellcotan}
\delta(1_{\mathbb{C}_p})=\pi^\ast([\alpha]).
\end{equation}
Using previous preliminaries, let us propose a generalization of a tangential cover in a Hitchin system as Corollary~\ref{corotan2}:
$$\delta\rho_\pi=\pi^\ast([\alpha]).$$

\begin{definition}\label{deftan2}
Let $\pi:\widehat{\mathfrak{R}}\to \mathfrak{R}$ be a Hitchin spectral cover. A {\em Hitchin tangency condition} is defined to be
\begin{equation}\label{hitcolax}
\delta(\rho_\pi)=\pi^\ast([\alpha]).
\end{equation}
\end{definition}

Note that for an elliptic soliton case, Equation~\eqref{hitcolax} is the same as
\begin{equation}\label{ell1}
\delta(1_{\mathbb{C}_p})=\pi^\ast\Big(\delta_q(1_{\mathbb{C}_q})\Big)
\end{equation}
where $\delta_q:\HH^0(X,\mathbb{C}_{q})\to\HH^1(X,\mathcal{O}_X)$ is induced from a short exact sequence on $X$,
$$0\to\mathcal{O}_X\to\mathcal{O}_X(q)\to\mathbb{C}_q\to0.$$
From Corollary~\ref{coclas} for a pointed Hitchin spectral cover $\pi:(\widehat{\mathfrak{R}},p)\to (\mathfrak{R},q)$, Equation~\eqref{hitcolax} is not equivalent to $\delta(\rho_\pi)=\pi^\ast\Big(\delta_q(1_{\mathbb{C}_q})\Big)$ but to
\begin{equation}\label{ell2}
\delta(\rho_\pi)=\pi^\ast\Big(\delta_{K_\mathfrak{R}}([\nu])+c\delta_{q}(1_{\mathbb{C}_q})\Big)
\end{equation}
for some $[\nu]\in\HH^0(\mathbb{C}_{K_\mathfrak{R}})$, since $[\alpha]=\delta_{K_\mathfrak{R}}([\nu])+c\delta_{q}(1_{\mathbb{C}_q})$. In particular, Corollary~\ref{corotan1} implies $c=0$. Hence, we may see that  Equation~\eqref{ell2} is an opposite generalization of Equation~\eqref{ell1}. On the other hands, Equation~\eqref{ell2} can be also seen as a generalization of Equation~\eqref{ell1}, since the canonical bundle $\K_X$ is trivial for an elliptic curve $X$. Hence, upon theses considerations it seems plausible to propose Definition~\ref{deftan2} as a generalization of the tangency condition of a Hitchin system.

\begin{remark}
In the Krichever theory \cite{kri77}, a Riemann surface $\widehat{\mathfrak{R}}$ with a divisor $D$ appears as a solution of the K-P equation. That is, the theta function associated with $\widehat{\mathfrak{R}}$ and the dynamics of divisor $D_t$ on $\widehat{\mathfrak{R}}$ give a solution of the K-P equation. Hence, a priori, there is no geometry of coverings involved. One method to define a particular subspace of space of all the K-P solutions is to use the reduction theory which we mention in Section~\ref{intropre}. When a theta function is reduced to theta functions of lower genera, the Riemann surface $\widehat{\mathfrak{R}}$ with a divisor $D$ appears as a cover of some Riemann surface, $\pi:\widehat{\mathfrak{R}}\to \mathfrak{R}$. The set of such coverings naturally defines a subspace of the space of all K-P solutions. Note that the configuration of the divisor $D$ on $\widehat{\mathfrak{R}}$ also plays a role in characterizing a solution in this case. For example, from \cite{kr80} we know that $\pi:\widehat{\mathfrak{R}}\to \mathfrak{R}$ with $D=\sum_{i=1}^{d}p_i$ is a matrix elliptic soliton if and only if all the $p_i$ for $i=1,\dots,d$ lie over one fiber, i.e., $p_i\in\pi^{-1}(q)$ for $i=1,\dots,d$. The $D$-tangency condition \cite{trei97} is an equivalent way to describe this method.

On the other hands, in the Hitchin theory \cite{hit87}, Hitchin spectral curves $\widehat{\mathfrak{R}}$ naturally appear as coverings of a Riemann surface $\mathfrak{R}$. One way to generalize the Krichever theory to this case is to use the generalized tangency condition of Definition~\ref{deftan2}. To get a subspace of the space of all Hitchin curves, what we proposed in Definition~\ref{deftan2} is, loosely speaking, equivalent to constructing a particular algebraic surface where a particular family of Hitchin spectral curves can be mapped into the surface. 
\end{remark}

Consider a short exact sequence
\begin{equation}\label{exsk1}
\xymatrix{
0\ar[r]&\mathcal{O}_{\widehat{\mathfrak{R}}}\ar[r]&\mathcal{O}_{\widehat{\mathfrak{R}}}\otimes\pi^\ast \mathcal{K}_\mathfrak{R}\ar[r]&\mathbb{C}_{\pi^{-1}(K_\mathfrak{R})}\ar[r]&0}
\end{equation}
and the induced long exact sequence:
\begin{equation}\label{lexsk1}
\xymatrix@R=1pt{
\cdots\ar[r]&\HH^0(\widehat{\mathfrak{R}},\pi^\ast\mathcal{K}_\mathfrak{R})\ar[r]&\HH^0(\widehat{\mathfrak{R}},\mathbb{C}_{\pi^{-1}(K_\mathfrak{R})})\ar[r]^(.55){\delta}&\HH^1(\widehat{\mathfrak{R}},\mathcal{O}_{\widehat{\mathfrak{R}}})\ar[r]&\cdots}
\end{equation}
From Corollary~\ref{corotan1}, we may let $\delta(\rho(\lambda))=[\alpha]$ for some local function $\lambda_\rho$ with the Laurent tail defined on the support of $K_\mathfrak{R}$. From Hitchin tangency condition~\eqref{hitcolax}, we have the zero class
\begin{equation}\label{fieq1}
[\delta\rho_\pi-\pi^\ast\Big(\delta(\rho(\lambda))\Big)]=[0]\in\HH^1(\widehat{\mathfrak{R}},\mathcal{O}_{\widehat{\mathfrak{R}}}).
\end{equation}
Hence, from long exact sequence~\eqref{lexsk1}, we may find a $w\in\HH^0(\widehat{\mathfrak{R}},\pi^\ast\mathcal{K}_\mathfrak{R})$ corresponding the zero class. Consequently, by construction the Laurent tail of $w+\pi^\ast(\lambda_\rho)$ at $\pi^{-1}(K_\mathfrak{R})$ is the residue section
$$\rho(w+\pi^\ast(\lambda_\rho))=\rho_\pi\in\HH^0(\widehat{\mathfrak{R}},\mathbb{C}_{\pi^{-1}(nK)}).$$
Let $(U_i,z_i)$ be a neighborhood of a point $q_i\in K_\mathfrak{R}=\sum_{i=1}^{2g-2}q_i$ with coordinate $z_i$. The defining equation of $\widehat{\mathfrak{R}}$ around $q_i$ is given by
$$0=R(k,z_i)=\prod_{j=1}^{n}(k+\frac{c_{i,j}}{z_i}+h_j(z_i))\text{ where }$$
$k$ is a function on $\widehat{\mathfrak{R}}$ given  by Theorem~\ref{thtan1} and $h_j(z_i)$ is a holomorphic function. Note that the residue section $\rho(k+\pi^\ast(\lambda_\rho))$ at $K_\mathfrak{R}=\sum_{i=1}^{2g-2}q_i$ is 
$$\big((c_{1,1}-\lambda_{1,1},\dots,c_{1,n}-\lambda_{1,n}),\dots,(c_{2g-2,1}-\lambda_{2g-2,1},\dots,c_{2g-2,n}-\lambda_{2g-2,n})\big).$$
Here
$$\big((\lambda_{1,1},\dots,\lambda_{1,n}),\dots,(\lambda_{2g-2,1},\dots,\lambda_{2g-2,n})\big).$$
is the residue section $\rho(\lambda_\rho)$ of $\lambda_\rho$. In fact, from Equation~\eqref{fieq1} it is not hard to see that
$$\rho(k+\pi^\ast(\lambda_\rho))=\rho_\pi.$$
Consequently, we see that $w=k$. Hence we have prove the following;

\begin{theorem}
Any Hitchin spectral cover satisfying tangency condition~\eqref{hitcolax} is a Hitchin tangential cover in Definition~\ref{hitandef1}. 
\end{theorem}

This shows that we may find an explicit local data of singularities of Hitchin tangential curves from the generalized tangency condition. For the elliptic soliton case, there are explicit defining equations of elliptic solitons and matrix elliptic solitons available to characterize singularities. See \cite{kr80, trei97} for details. It will be also interesting to see how theses singularities look like for the case of sub-linear systems of Hitchin spectral curves. That is, an investigation about the characteristic of singularities of Hitchin spectral curves associated with vector bundles {\cite{hit87} with gauge groups, for example, $\mathbf{SO}(2m,\mathbb{C}),\mathbf{SP}(m,\mathbb{C}), \mathbf{SO}(2m+1,\mathbb{C})$, and $\mathbf{G}_2$, etc, would be an interesting problem. Therefore, we will revisit this investigation and deal with theses cases for somewhere else in the future.

Another project to which we want to draw the attention of the audience is the finiteness of Hitchin hyperelliptic tangential covers over a hyperelliptic curve. The main and remarkable application of the Treibich-Verdier theory was to provide the proof of the finiteness of ``hyperelliptic tangential covers'' over a fixed elliptic curve (see \emph{p.462} in \cite{tv90}). 
In the sense of tangency condition in this paper, we can also formulate a similar statement: In our setting, the base curve $\mathfrak{R}$ should be a hyperelliptic curve and we should consider the moduli of Hitchin hyperelliptic tangential covers over a hyperelliptic curve $\mathfrak{R}$. The finiteness of the space can be obtained by mimicking similar procedures in \cite{tv90}. However, the main difficulty to follow the procedures in \cite{tv90} comes from the way of constructing the ruled surface $\mathfrak{S}$. Note that the method to construct a ruled surface in \cite{tv90} is different from ours. Hence, it is not possible to follow them directly to prove the desired result in our setting. In order to remove this drawback, we should provide a different way to construct the surface $\mathfrak{S}$ to make a parallel approach as in \cite{tv90}. We will give the construction and the proof elsewhere \cite{kim111}.

\end{document}